\documentclass[a4paper,10pt]{amsart}
\usepackage{amsmath,amsfonts,amsthm,amssymb}
\usepackage[utf8]{inputenc}

\title{A structure theorem for sets of small popular doubling, revisited}
\author{Przemys\l aw Mazur}
\address{Mathematical Institute, Radcliffe Observatory Quarter, Woodstock Road, Oxford OX2 6GG, United Kingdom}
\email{przemyslaw.mazur@maths.ox.ac.uk}

\newtheorem{thm}{Theorem}[section]
\newtheorem{prp}[thm]{Proposition}
\newtheorem{lm}[thm]{Lemma}

\theoremstyle{definition}

\newtheorem{cor}[thm]{Corollary}

\renewcommand{\le}{\leqslant}
\renewcommand{\ge}{\geqslant}
\renewcommand{\epsilon}{\varepsilon}
\renewcommand{\phi}{\varphi}

\begin{document}
 
 \begin{abstract}
  We prove that every set $A\subset\mathbb{Z}/p\mathbb{Z}$ with $\mathbb{E}_x\min(1_A*1_A(x),t)\le (2+\delta)t\mathbb{E}_x 1_A(a)$ is very close to an arithmetic progression. Here $p$ stands for a large prime and $\delta,t$ are small real numbers. This shows that the Vosper theorem is stable in the case of a single set.
 \end{abstract}

 \maketitle
 
 \section{Introduction}
 
 In the recent paper \cite{maz15} we proved a structure theorem for sets of integers having small popular doubling. We were aiming to extend this theorem to make it also work for sets of residue classes modulo a prime. Unfortunately we were unable to achieve this using the methods of that paper. In this paper we prove that missing statement using entirely different methods. To be more specific, our goal is to prove the following statement.
 
 \begin{thm}\label{main}
  Let $0<\alpha_1<\alpha_2<\frac{1}{4}$ and $\eta>0$. Then there exist positive real numbers $\delta_0=\delta_0(\alpha_1,\alpha_2,\eta)$, $C=C(\alpha_1,\alpha_2,\eta)$ and $p_0=p_0(\alpha_1,\alpha_2,\eta)$ with the following properties. Let $p>p_0$ be a prime and let $A\subset\mathbb{Z}/p\mathbb{Z}$ be a set. Suppose that the density $\alpha=\frac{|A|}{p}$ satisfies $\alpha_1<\alpha<\alpha_2$. Furthermore, suppose that
  \begin{equation*}
   \mathbb{E}_{x}\min(1_A*1_A(x),t)\le(2+\delta)\alpha t
  \end{equation*}
  for some numbers $\delta\in (0,\delta_0)$ and $t\in (0,t_0(\alpha_1,\alpha_2,\eta,\delta))$. Then there is an arithmetuc progression $P$ with $|P|\le (1+(1+\eta)\delta)\alpha p$ and $|A\setminus P|\le C(\delta\alpha)^{1/2}p.$
 \end{thm}

 To fix the notation, let us use the Haar probability measure on all groups appearing in this paper. That means that the symbol $\mathbb{E}_x$ used above is just a shorthand for $\frac{1}{p}\sum_x$, and by $f*g(x)$ we mean $\mathbb{E}_yf(y)g(x-y)$.
 
 The dependences in the statement of Theorem \ref{main} look rather complicated, let us justify them a little bit. If we were dealing with sets satisfying just $|A+A|\le (2+\delta)\alpha p$, then the correct bound for the size of $P$ would be $|P|=|A+A|-|A|+1=(1+\delta)\alpha p+1$ (see \cite{sz09} for details). The parameter $\eta$ indicates that we can make as small error as we like, even in terms of $\delta$, but to achieve that the popularity parameter $t$ has to be sufficiently small, in terms of both $\eta$ and $\delta$. Ideally we would like to conclude that $|P|\le((1+\delta)\alpha+O(t))p$ (see \cite{maz15}), but with our methods we are unable to achieve that.
 
 
 \section{Proof of the main result}
 
 Let us start with fixing all the parameters and the set $A\subset\mathbb{Z}/p\mathbb{Z}$ satisfying the assumption. We intend to apply the arithmetic regularity lemma (Theorem \ref{arl}) to the function $f=1_A$. Let $\epsilon>0$ and $\mathcal{F}$ be a growth function to be specified later. Then we can write $f=f_\mathrm{str}+f_\mathrm{sml}+f_\mathrm{unf}$, as in the statement of Theorem \ref{arl}. Let us first get rid of the function $f_\mathrm{unf}$.
 
 \begin{lm}
  Let $g,h:\mathbb{Z}/p\mathbb{Z}\to\mathbb{C}$ be functions. Then the following inequality holds: \begin{equation*}
   \|g*h\|_2\le\|g\|_{U^2}\|h\|_{U^2}
  \end{equation*}
 \end{lm}
 
 \begin{proof}
  Using Parseval's identity and the relation between the convolution and Fourier transform we get
  \begin{equation*}
   \|g*h\|_2^2=\sum_r |\widehat{g*h}(r)|^2=\sum_r|\widehat{g}(r)|^2|\widehat{h}(r)|^2.
  \end{equation*}
  On the other hand, we know that
  \begin{equation*}
   \|g\|_{U^2}^4=\sum_r|\widehat{g}(r)|^4\qquad\text{and}\qquad\|h\|_{U^2}^4=\sum_r|\widehat{h}(r)|^4.
  \end{equation*}
  The inequality is then equivalent to the Cauchy-Schwarz inequality in the following form:
  \begin{equation*}
   \left(\sum_r|\widehat{g}(r)|^2|\widehat{h}(r)|^2\right)^2\le\left(\sum_r|\widehat{g}(r)|^4\right)\left(\sum_r|\widehat{h}(r)|^4\right).
  \end{equation*}

 \end{proof}

 \begin{cor}
  Let $1_A=f_\mathrm{str}+f_\mathrm{sml}+f_\mathrm{unf}$ as above. Then the following inequality holds:
  \begin{equation*}
   \mathbb{E}_{x}\min((f_\mathrm{str}+f_\mathrm{sml})*(f_\mathrm{str}+f_\mathrm{sml}),t)\le (2+\delta)\alpha t+\frac{2}{(\mathcal{F}(M))^{1/2}}.
  \end{equation*}

 \end{cor}

 \begin{proof}
  First of all, since for all characters $\chi$ we have $|\widehat{f_\mathrm{unf}}(\chi)|=|\langle f_\mathrm{unf},\chi\rangle|\le\frac{1}{\mathcal{F}(M)}$, we can estimate the $U^2$ norm of $f_\mathrm{unf}$ as
  \begin{multline*}
   \|f_\mathrm{unf}\|_{U^2}^4=\sum_{\chi}|\widehat{f_\mathrm{unf}}(\chi)|^4\le\frac{1}{(\mathcal{F}(M))^2}\sum_{\chi}|\widehat{f_\mathrm{unf}}(\chi)|^2=\\
   =\frac{\|f_\mathrm{unf}\|_2^2}{(\mathcal{F}(M))^2}\le\frac{\|f_\mathrm{unf}\|_\infty^2}{(\mathcal{F}(M))^2}=\frac{1}{(\mathcal{F}(M))^2}.
  \end{multline*}
  Therefore for any function $g:\mathbb{Z}/p\mathbb{Z}\to\mathbb{C}$ with $\|g\|_\infty\le 1$ the lemma above gives
  \begin{equation*}
   \|f_\mathrm{unf}*g\|_1\le\|f_\mathrm{unf}*g\|_2\le\|f_\mathrm{unf}\|_{U^2}\|g\|_{U^2}\le\|f_\mathrm{unf}\|_{U^2}\|g\|_\infty\le\frac{1}{(\mathcal{F}(M))^{1/2}}.
  \end{equation*}
  Applying this to the functions $g=1_A$ and $g=f_\mathrm{str}+f_\mathrm{sml}$ and using triangle inequality we get
  \begin{equation*}
   \|1_A*1_A-(f_\mathrm{str}+f_\mathrm{sml})*(f_\mathrm{str}+f_\mathrm{sml})\|_1\le\frac{2}{(\mathcal{F}(M))^{1/2}}.
  \end{equation*}
  Now we use an easy-to-check inequality $|\min(a,t)-\min(b,t)|\le|a-b|$ for $a=1_A*1_A(x)$ and $b=(f_\mathrm{str}+f_\mathrm{sml})*(f_\mathrm{str}+f_\mathrm{sml})(x)$ for any $x\in\mathbb{Z}/p\mathbb{Z}$. Combining them with another instance of triangle inequality yields
  \begin{equation*}
   |\mathbb{E}_{x}\min(1_A*1_A(x),t)-\mathbb{E}_{x}\min((f_\mathrm{str}+f_\mathrm{sml})*(f_\mathrm{str}+f_\mathrm{sml}),t)|\le\frac{2}{(\mathcal{F}(M))^{1/2}},
  \end{equation*}
  which gives the result.
 \end{proof}

 We managed to remove $f_\mathrm{unf}$ from our considerations, now it is time for $f_\mathrm{sml}$. To deal with this. let $\lambda$ be a small quantity to be specified and let $B=\{x\in\mathbb{Z}/p\mathbb{Z}:\|\phi(x)\|\le\frac{\lambda}{2M}\}$ be a Bohr set. Recall that $\phi$ is the homomorphism used to construct $f_\mathrm{str}$, for the definition of $\|\phi(x)\|$, see the appendix. Now let
 \begin{equation*}
  C=\{x\in\mathbb{Z}/p\mathbb{Z}:f_\mathrm{str}(x)\ge\lambda\quad\text{and}\quad\mathbb{E}_{y\in B}|f_\mathrm{sml}(x+y)|^2\le\epsilon\}.
 \end{equation*}
 Intuitively, we take all elements where $f_\mathrm{srt}$ is somewhat large and where $f_\mathrm{sml}$ is too small to destroy that. First of all, let us estimate the size of $C$. The set $C'=\{x\in\mathbb{Z}/p\mathbb{Z}:f_\mathrm{str}(x)\ge\lambda\}$ has size at least $\sum_x f_\mathrm{str}(x)-\lambda p$, since
 \begin{equation*}
  \sum_x f_\mathrm{str}(x)=\sum_{x\in C'} f_\mathrm{str}(x)+\sum_{x\not\in C'} f_\mathrm{str}(x)\le |C'|+\lambda p.
 \end{equation*}
 On the other hand, the set $C''=\{x\in\mathbb{Z}/p\mathbb{Z}:\mathbb{E}_{y\in B}|f_\mathrm{sml}(x+y)|^2>\epsilon\}$ has size at most $\epsilon p$, because
 \begin{equation*}
  \epsilon^2\ge\mathbb{E}_x|f_\mathrm{sml}(x)|^2=\mathbb{E}_x(\mathbb{E}_{y\in B}|f_\mathrm{sml}(x+y)|^2)\ge\frac{\epsilon|C''|}{p}.
 \end{equation*}
 Therefore the size of $C$ can be estimated as $|C|=|C'\setminus C''|\ge |C'|-|C''|\ge\sum_x f_\mathrm{str}(x)-(\lambda+\epsilon)p$. To make it more explicit, note that from the construction of $f_\mathrm{unf}$ we see that $\mathbb{E}_x f_\mathrm{unf}(x)=0$. That leads to $\mathbb{E}_x (f_\mathrm{str}+f_\mathrm{sml})(x)=\mathbb{E}_x 1_A(x)=\alpha$; combining it with $|\mathbb{E}_x f_\mathrm{sml}(x)|\le\|f_\mathrm{sml}\|_1\le\|f_\mathrm{sml}\|_2\le\epsilon$ we get $\mathbb{E}_x f_\mathrm{str}(x)\ge\alpha-\epsilon$. In the end it means that $|C|\ge(\alpha-2\epsilon-\lambda)p$.
 
 Now it is time to see the reason why we defined the set $C$ in this way. To see this, let $x_1,x_2\in C$ and consider four functions: $f_1,f_2,g_1,g_2:B\to\mathbb{C}$ defined as:
 \begin{equation*}
  f_i(x)=f_\mathrm{str}(x_i+(-1)^ix),\qquad g_i(x)=f_\mathrm{sml}(x_i+(-1)^ix).
 \end{equation*}
 Since $f_\mathrm{str}+f_\mathrm{sml}$ is a nonnegative function, we have the inequality
 \begin{multline*}
  (f_\mathrm{str}+f_\mathrm{sml})*(f_\mathrm{str}+f_\mathrm{sml})(x_1+x_2)\ge\\
  \ge\mathbb{E}_x(f_\mathrm{str}+f_\mathrm{sml})(x_1-x)(f_\mathrm{str}+f_\mathrm{sml})(x_2+x)1_B(x)=\\
  =\frac{|B|}{p}\langle f_1+g_1,f_2+g_2\rangle.
 \end{multline*}
 Now from the Lipschitz nature of $F$ and the definitions of $B$ and $C$ we know that $f_1(x),f_2(x)\ge\frac{\lambda}{2}$ for all $x\in B$, which leads to $\langle f_1,f_2\rangle\ge\frac{\lambda^2}{4}$. Moreover since $\|f_i\|_2\le\|f_i\|_\infty\le 1$ and $\|g_i\|_2\le\sqrt{\epsilon}$ (by the definition of $C$), we also have $|\langle f_1,g_2\rangle|,|\langle f_2,g_1\rangle|\le\sqrt{\epsilon}$ and $|\langle g_1,g_2\rangle|\le\epsilon$. Combining all the inequalities together we get $\langle f_1+g_1,f_2+g_2\rangle\ge\frac{\lambda^2}{4}-2\sqrt{\epsilon}-\epsilon$. Also, from the properties of Bohr sets (see for example \cite{tv06}) we know that $\frac{|B|}{p}\ge(\frac{\lambda}{2M})^{\dim B}\ge(\frac{\lambda}{2M})^M$. Therefore if only $t\le(\frac{\lambda}{2M})^M(\frac{\lambda^2}{4}-2\sqrt{\epsilon}-\epsilon)$, we have just proved that $(f_\mathrm{str}+f_\mathrm{sml})*(f_\mathrm{str}+f_\mathrm{sml})(x)\ge t$ for all $x\in C+C$. Since
 \begin{equation*}
  \mathbb{E}_{x}\min((f_\mathrm{str}+f_\mathrm{sml})*(f_\mathrm{str}+f_\mathrm{sml}),t)\le (2+\delta)\alpha t+\frac{2}{(\mathcal{F}(M))^{1/2}},
 \end{equation*}
 we know that in this case we have $|C+C|\le((2+\delta)\alpha+\frac{2}{t(\mathcal{F}(M))^{1/2}})p$.
 
 The main term of the above expression is $2\alpha p$, while the main term of the expression bounding the size of $C$ is $\alpha p$. Therefore if the error terms are sufficiently small, we can make use of Serra-Z\'emor Theorem (proven in \cite{sz09}) and conclude that the set $C$ is contained in an arithmetic progression $P\subset\mathbb{Z}/p\mathbb{Z}$ of size $|P|=|C+C|-|C|+1$. We can assume without loss of generality that $|P|\ge |A|$ as we can extend $P$ if necessary. We will come back later to the conditions that must be satisfied, let us now proceed with the proof.
 
 We will examine how the progression $P$ is related to the set $A$. First of all, since $C\subset P$, we know that there can only be $\epsilon p$ elemets $x$ outside $P$ for which $f_\mathrm{str}\ge\lambda$. Therefore we have the inequality
 \begin{equation*}
  \mathbb{E}_x\max((f_\mathrm{str}-1_P)(x),0)\le\epsilon+\lambda.
 \end{equation*}
 This means that we also have $\mathbb{E}_x\max((1_P-f_\mathrm{str})(x),0)\le\mathbb{E}_x(1_P-f_\mathrm{str})(x)+\epsilon+\lambda$. Adding those two inequalities we get $\|1_P-f_\mathrm{str}\|_1\le\mathbb{E}_x(1_P-f_\mathrm{str})(x)+2(\epsilon+\lambda)$. The last quantity is then an upper bound for the absolute value of the difference of corresponding Fourier coefficients of $1_P$ and $f_\mathrm{str}$. In other words, $|\langle 1_P-f_\mathrm{str},\chi\rangle|\le\mathbb{E}_x(1_P-f_\mathrm{str})(x)+2(\epsilon+\lambda)$ for each character $\chi$. On the other hand, we know that
 \begin{equation*}
  |\langle 1_A-f_\mathrm{str},\chi\rangle|\le|\langle f_\mathrm{sml},\chi\rangle|+|\langle f_\mathrm{unf},\chi\rangle|\le\epsilon+\frac{1}{\mathcal{F}(M)}
 \end{equation*}
 for each character $\chi$. By triangle inequality it means that for every character $\chi$ the following holds:
 \begin{multline*}
  |\langle 1_P-1_A,\chi\rangle|\le\mathbb{E}_x(1_P-f_\mathrm{str})(x)+3\epsilon+2\lambda+\frac{1}{\mathcal{F}(M)}\le\\
  \le\mathbb{E}_x(1_P-1_A)(x)+4\epsilon+2\lambda+\frac{1}{\mathcal{F}(M)}.
 \end{multline*}
 Recall now that $P$ is an arithmetic progression, so one of its (non-trivial) Fourier coefficients is as large as it could possibly be, more precisely there exists $\chi_1$ with \begin{equation*} 
  |\widehat{1_P}(\chi_1)|=|\langle 1_P,\chi_1\rangle|=\frac{\sin\left(\frac{(|P|-1)\pi}{p}\right)}{p\sin\left(\frac{\pi}{p}\right)}.
 \end{equation*}
 Now let $z$ be a unit complex number satisfying $z\widehat{1_P}(\chi_1)=|\widehat{1_P}(\chi)|$. Since $1_P-1_A=2\cdot 1_P-1_{A\cap P}-1_{A\cup P}$, we have the following lower bound:
 \begin{multline*}
  |\langle 1_P-1_A,\chi\rangle|\ge\Re(z\cdot\langle 2\cdot 1_P-1_{A\cap P}-1_{A\cup P},\chi\rangle)\ge\\
  \ge\frac{2\sin\left(\frac{(|P|-1)\pi}{p}\right)-\sin\left(\frac{(|A\cap P|-1)\pi}{p}\right)-\sin\left(\frac{(|A\cup P|-1)\pi}{p}\right)}{p\sin\left(\frac{\pi}{p}\right)}.
 \end{multline*}
 We can rearrange the numerator of the last expression as follows
 \begin{multline*}
  2\sin\left(\frac{(|P|-1)\pi}{p}\right)-\sin\left(\frac{(|A\cap P|-1)\pi}{p}\right)-\sin\left(\frac{(|A\cup P|-1)\pi}{p}\right)=\\
  =4\sin\left(\frac{|P\setminus A|\cdot\pi}{2p}\right)\sin\left(\frac{|A\setminus P|\cdot\pi}{2p}\right)\sin\left(\frac{(|A|+|P|-2)\pi}{2p}\right)+\\
  +2\sin\left(\frac{(|P|-|A|)\pi}{2p}\right)\cos\left(\frac{(|A|+|P|-2)\pi}{2p}\right).
 \end{multline*}
 Now we are almost ready to estimate the size $|A\setminus P|$. First of all, since $|P|\ge |A|$, the last summand is positive and can be discarded, leaving us with the inequality
 \begin{equation*}
  \mathbb{E}_x(1_P-1_A)(x)+4\epsilon+2\lambda+\frac{1}{\mathcal{F}(M)}\ge\frac{4\sin\left(\frac{|P\setminus A|\cdot\pi}{2p}\right)\sin\left(\frac{|A\setminus P|\cdot\pi}{2p}\right)\sin\left(\frac{(|A|+|P|-2)\pi}{2p}\right)}{p\sin\left(\frac{\pi}{p}\right)}.
 \end{equation*}
 
 If only $4\epsilon+2\lambda+\frac{1}{\mathcal{F}(M)}\le(1-\eta)\delta\alpha$ and $|P|\le(1+(1+\eta)\delta)\alpha p$, we have that the left hand side is bounded by $2\delta\alpha$.
 On the other hand, if we had $|A\setminus P|>C(\delta\alpha)^{1/2}p$, then the same would hold for $|P\setminus A|$. Knowing the behaviour of sine around $0$, we would argue that the first two factors in the numerator are at least $C'(\delta\alpha)^{1/2}$ for some other constant $C'$. But the last factor is bounded away from $0$ (as $|A|$ and $|P|$ are bounded away from both $0$ and $\frac{p}{2}$) and the denominator is around $\pi$ (w.l.o.g. $>3$), so this would contradict our inequality. In the end we need to have $|A\setminus P|\le C(\alpha\delta)^{1/2}p$.
 
 The Theorem is now proven up to checking that we can choose all the constants to make the calculations work. First of all, we would like to use the Serra-Zemor Theorem for the set $C$. We had $|C+C|\le((2+\delta)\alpha+\frac{2}{t(\mathcal{F}(M))^{1/2}})p$ and $|C|\ge(\alpha-\lambda-\epsilon)p$. To make sure that $|C+C|<(2+10^{-4})|C|$ we want to require $\lambda,\epsilon<10^{-6}\alpha$, $\delta<10^{-6}$ and $t(\mathcal{F}(M))^{1/2}>10^6\alpha^{-1}$ (say). Then, we would like to have $|C+C|-|C|+1=|P|\le(1+(1+\eta)\delta)\alpha p$. This rearranges to
 \begin{equation*}
  \frac{2}{t(\mathcal{F}(M))^{1/2}}+\epsilon+\lambda<\eta\delta\alpha.
 \end{equation*}
 For this it would be enough if $\lambda,\epsilon, \frac{1}{t(\mathcal{F}(M))^{1/2}}\le\frac{\eta\delta\alpha}{4}$. Moreover, we need $t\le(\frac{\lambda}{2M})^M(\frac{\lambda^2}{4}-2\sqrt{\epsilon}-\epsilon)$. This suggests setting $\epsilon=\frac{\lambda^4}{256}$ and requiring $t\le(\frac{\lambda}{2M})^M\cdot\frac{\lambda^2}{16}$. We have just listed all the requirements and now the strategy is as follows. Set $\lambda=\frac{\eta\delta\alpha}{4}$ (we can freely assume $\eta<10^{-6}$ to make sure that $\lambda<10^{-6}\alpha$) and $\epsilon=\frac{\lambda^4}{256}$. Now the only thing is to make sure that $\frac{4}{\eta\delta\alpha(\mathcal{F}(M))^{1/2}}\le t\le\frac{\lambda^2}{16}(\frac{\lambda}{2M})^M$. This might seem impossible, as the upper bound on $M$ depends of $\mathcal{F}$ and we might not be able to fit into the correct range. The solution to this problem is the following: suppose that the above inequalities hold for some other number $t'$. Then the entire argument is correct assuming that the 
initial inequality describing popular doubling of $A$ holds with parameter $t'$ instead of $t$. A similar argument to \cite[Corollary 3.5]{maz15} shows that this is indeed the case for any $t'\ge t$. This suggest the following strategy:
 \begin{itemize}
  \item given $\alpha_1,\alpha_2,\eta$, choose $\delta_0>0$ so that $(1+(1+\eta)\delta_0)\alpha_2<\frac{1}{2}$ (to make sure all the hyptheses of Serra-Zemor Theorem are satisfied),
  \item given $\delta\in(0,\delta_0)$, set $\lambda=\frac{\eta\delta\alpha_1}{4}$ and $\epsilon=\frac{\lambda^4}{256}$,
  \item define $\mathcal{F}(M)=\frac{2^{12}}{(\eta\delta\alpha_1\lambda^2)^2}(\frac{2M}{\lambda})^{2M}$ and apply the arithmetic regularity lemma to get an upper bound $M\le M_0$,
  \item set $t_0=\frac{\lambda^2}{16}(\frac{\lambda}{2M_0})^{M_0}$.
 \end{itemize}
 Then we can find $t'\ge t_0$ with the postulated properties. Since $t\le t_0$, we also have $t\le t'$, as required. That ends the proof of Theorem \ref{main}.

 \appendix
 
 \section{Arithmetic regularity lemma}
 
 In the appendix we give a self-contained proof of the arithmetic regularity lemma for $U^2$ norm. The lemma was proven in full generality (i.e. for $U^k$ norm for arbitrary $k$) by Green and Tao in \cite{gt10}. There is also an exposition by Eberhard of the $U^2$ case. Unfortunately both of them have a feature that is a disadvantage for us, namely they deal with functions defined on $\{1,\ldots,N\}$ rather than $\mathbb{Z}/p\mathbb{Z}$. As a result the ``structured part'' obtained there comes from a Lipschitz function defined on $[0,1]\times\mathbb{Z}/q\mathbb{Z}\times\mathbb{T}^d$ (in the $U^2$ case). However, as we work over a cyclic group of prime order, in our setting everything is periodic $\pmod{p}$ and there is no room either for non-periodic behaviour (such as $[0,1]$), or periodic behaviour modulo other numbers. Therefore we are aiming for a slightly different statement of the regularity lemma, but the methods of proof remain the same.
 
 Let us by an \emph{pre-character} on a group $G$ mean a homomorphism $\phi:G\to\mathbb{T}$ and by a \emph{character} a homomorphism $\chi:G\to\{z\in\mathbb{C}\ :\ |z|=1\}$. Of course there is one-to-one correspondence between those, given by the equation $\chi=e^{2\pi i\phi}$. This terminology is by no means standard and is used only for the purpose of this paper.
 
 Before we start, let us fix some notation. For any set $\Gamma$ of pre-characters on $\mathbb{Z}/p\mathbb{Z}$ and any positive integer $n$, we define a partition $\mathcal{B}=\mathcal{B}(\Gamma, n)$ of $\mathbb{Z}/p\mathbb{Z}$ into cells. Intuitively, $\mathcal{B}$ corresponds to the partition of the torus $\mathbb{T}^\Gamma$ into $n^{|\Gamma|}$ cubes of side length $\frac{1}{n}$. More formally, two points $x,y\in\mathbb{Z}/p\mathbb{Z}$ belong to the same cell if $\phi(x),\phi(y)\in [\frac{k_\phi}{n},\frac{k_\phi+1}{n})\subset\mathbb{T}$ for some $k_\phi\in\mathbb{Z}$ for all $\phi\in\Gamma$. Note that if $\Gamma\subset\Gamma'$ and $n|n'$, then $\mathcal{B}(\Gamma', n')$ is a refinement of $\mathcal{B}(\Gamma,n)$, i.e. each cell of the former is a union of cells of the latter.
 
 For any function $f:\mathbb{Z}/p\mathbb{Z}\to\mathbb{C}$ and any partition $\mathcal{B}=\mathcal{B}(\Gamma,n)$ define the conditional expectation $\mathbb{E}(f|\mathcal{B})$ in the standard way, i.e. $\mathbb{E}(f|\mathcal{B})(x)$ is the average of $f$ on the cell of $\mathcal{B}$ containing $x$. In other words, $\mathbb{E}(f|\mathcal{B})$ is just the orthogonal projection of $f$ onto the space of all $\mathcal{B}$-measurable functions (constant on every cell of $\mathcal{B}$). Note that if $\mathcal{B}'$ is a refinement of $\mathcal{B}$ then $\mathbb{E}(\mathbb{E}(f|\mathcal{B}')|\mathcal{B})=\mathbb{E}(f|\mathcal{B})$ and more generally $\mathbb{E}(f\cdot\mathbb{E}(g|\mathcal{B})|\mathcal{B'})=\mathbb{E}(f|\mathcal{B}')\mathbb{E}(g|\mathcal{B})$.
 
 \begin{lm}
  Let $n>0$ and let $\phi:\mathbb{Z}/p\mathbb{Z}\to\mathbb{T}$ be an pre-character and let $\chi=e^{2\pi i\phi}$. Let $\Gamma$ be a set of characters containing $\phi$ and let $\mathcal{B}=\mathcal{B}(\Gamma,n)$. Suppose that $f:\mathbb{Z}/p\mathbb{Z}\to\mathbb{C}$ is a function with $\|f\|_\infty\le 1$. Then
  \begin{equation*}
   \left|\langle f-\mathbb{E}(f|\mathcal{B}),\chi\rangle\right|\le\frac{2\pi}{n}.
  \end{equation*}
 \end{lm}

 \begin{proof}
  The key idea is that $\chi$ is almost constant on each cell of $\mathcal{B}$. More precisely, by the properties of orthogonal projections we have
  \begin{equation*}
   \langle f-\mathbb{E}(f|\mathcal{B}),\chi\rangle=\langle f,\chi-\mathbb{E}(\chi|\mathcal{B})\rangle.
  \end{equation*}
  But since $\phi\in\Gamma$, the function $\chi-\mathbb{E}(\chi|\mathcal{B})$ is bounded pointwise by $|1-e^{2\pi i/n}|\le\frac{2\pi}{n}$ and the claim follows.
 \end{proof}
 
 \begin{cor}\label{uniformization}
  Let $\delta>0$ and let $f:\mathbb{Z}/p\mathbb{Z}\to\mathbb{C}$ with $\|f\|_\infty\le 1$. Then there exists a set $\Gamma$ of pre-characters of size $|\Gamma|\le\frac{4}{\delta^2}$ and $n\le\frac{16}{\delta}$ such that for $\mathcal{B}=\mathcal{B}(\Gamma,n)$ and any character $\chi$ we have
  \begin{equation}\label{fou_coeff}
   |\langle f-\mathbb{E}(f|\mathcal{B}),\chi\rangle|\le\delta.
  \end{equation}

 \end{cor}

 \begin{proof}
  Define $n=\lceil\frac{4\pi}{\delta}\rceil$ and build the set $\Gamma$ iteratively, at the beginning $\Gamma=\emptyset$. At each stage we ask if the inequality \eqref{fou_coeff} is satisfied for every character. If so, we finish our procedure, otherwise we take a character $\chi$ for which the inequality fails and add the corresponding pre-character $\phi$ to $\Gamma$. Let $\mathcal{B}=\mathcal{B}(\Gamma,n)$ and $\mathcal{B}'=\mathcal{B}(\Gamma\cup\{\phi\},n)$. By the previous lemma we know that
  \begin{equation*}
   \left|\langle f-\mathbb{E}(f|\mathcal{B}'),\chi\rangle\right|\le\frac{2\pi}{n}\le\frac{\delta}{2}.
  \end{equation*}  
  Combining this with the initial assumption on $\chi$ and triangle inequality gives
  \begin{equation*}
   \left|\langle \mathbb{E}(f|\mathcal{B}')-\mathbb{E}(f|\mathcal{B}),\chi\rangle\right|\ge\frac{\delta}{2}.
  \end{equation*}
  Now we use Cauchy-Schwarz and the fact that $\mathbb{E}(f|\mathcal{B})$ is an orthogonal projection of $\mathbb{E}(f|\mathcal{B}')$ (as $\mathcal{B}'$ is a refinement of $\mathcal{B}$):
  \begin{equation*}
   \|\mathbb{E}(f|\mathcal{B}')\|_2^2-\|\mathbb{E}(f|\mathcal{B})\|_2^2=\|\mathbb{E}(f|\mathcal{B}')-\mathbb{E}(f|\mathcal{B})\|_2^2\ge |\langle \mathbb{E}(f|\mathcal{B}')-\mathbb{E}(f|\mathcal{B}),\chi\rangle|^2\ge\frac{\delta^2}{4}.
  \end{equation*}
  In other words, it means that adding to $\Gamma$ the pre-character corresponding to $\chi$ increases the value of $\|\mathbb{E}(f|\mathcal{B})\|_2^2$ by at least $\frac{\delta^2}{4}$. Since this quantity can only take values between $0$ and $1$, this procedure must terminate in at most $\frac{4}{\delta^2}$ steps. In the end we get a set $\Gamma$ of size at most $\frac{4}{\delta^2}$ satisfying the inequality \eqref{fou_coeff} for each character $\chi$.
  
  It remains to check that the bound on $n$ is correct. For that we can freely assume $\delta\le 1$, which implies $\frac{4\pi}{\delta}\ge4\pi>\frac{25}{2}$. However, for any such number $x$ we have the bound $\lceil x \rceil\le\frac{14}{13}x\le\frac{16}{4\pi}x$.
 \end{proof}

 The corollary above says that we can get rid of any large Fourier coefficients using only projections of bounded complexity. The heart of the arithmetic regularity lemma is to iterate this argument. Before we do that, le us explain what a growth function is. A \emph{growth function} is simply an increasing function $\mathcal{F}:(0,+\infty)\to (0,+\infty)$, typically describing how large we need one parameter to be in terms of another parameter. In most applications one can think of $\mathcal{F}$ as of an exponential function $x\mapsto C_1 e^{C_2x}$.
 
 \begin{prp}[arithmetic regularity lemma: baby version]
  Let $\epsilon>0$ and let $\mathcal{F}$ be a growth function. Then there exixts a real number $M_0=M_0(\epsilon,\mathcal{F})>0$ for which the following statement is true. Let $p$ be a prime and $f:\mathbb{Z}/p\mathbb{Z}\to [0,1]$ be a function. Then there exists a number $0<M\le M_0$ and a decomposition
  \begin{equation*}
   f=f_\mathrm{str}+f_\mathrm{sml}+f_\mathrm{unf}
  \end{equation*}
  satisfying the following properties:
  \begin{itemize}
   \item $f_\mathrm{str}=\mathbb{E}(f|\mathcal{B})$, where $\mathcal{B}=\mathcal{B}(\Gamma,n)$ for some set $\Gamma$ of pre-characters and some positive integer $n$ with $|\Gamma|,n\le M$ ($f_\mathrm{str}$ is structured),
   \item $\|f_\mathrm{sml}\|_2\le\epsilon$ ($f_\mathrm{sml}$ is small),
   \item $|\langle f_\mathrm{unf},\chi\rangle|\le\frac{1}{\mathcal{F}(M)}$ for every character $\chi$ ($f_\mathrm{unf}$ is $U^2$-uniform).
   \item $f_\mathrm{str}$ and $f_\mathrm{str}+f_\mathrm{sml}$ both take values in $[0,1]$.
  \end{itemize}
 \end{prp}
 
 \begin{proof}
  We will again use an iterative procedure. At the beginning, let $\mathcal{B}$ be the trivial partition, corresponding to $\Gamma=\emptyset$ and $n=1$. At each stage, we set $M=\max(|\Gamma|,n)$ and then apply Corollary \ref{uniformization} with parameter $\delta=\frac{1}{\mathcal{F}(M)}$ to the function $f-\mathbb{E}(f|\mathcal{B})$. This way we get a set $\Gamma'$ and an integer $n'$, both bounded in terms of $M$ and $\mathcal{F}$. In fact, we need a slightly modified version of this result; to ensure that $\Gamma\subset\Gamma'$ and $n|n'$, we take at the beginning $n\cdot\lceil\frac{4\pi}{\delta}\rceil$ instead of $\lceil\frac{4\pi}{\delta}\rceil$ and $\Gamma$ istead of the empty set. This does not affect the boundedness of the final parameters, we still have the bounds of the shape $|\Gamma'|,n\le\mathcal{F}'(M)$ for some growth function $\mathcal{F}'$ depending only on $\mathcal{F}$. After applying this procedure we wish to set $f_\mathrm{str}=\mathbb{E}(f|\mathcal{B})$, $f_\mathrm{sml}=\mathbb{E}(f|\
mathcal{B})'-\mathbb{E}(f|\mathcal{B})$ and $f_\mathrm{unf}=f-\mathbb{E}(f|\mathcal{B})'$, where $\mathcal{B}'=\mathcal{B}(\Gamma',n')$. All the required conditions are clearly satisfied except one: it might happen that $\|f_\mathrm{sml}\|_2>\epsilon$. To take care of it, we use iteration: if this actually happened, set new $\Gamma:=\Gamma'$ and $n:=n'$. Again, we can argue that since $\mathcal{B}'$ is a refinement of $\mathcal{B}$, we have
  \begin{equation*}
   \|\mathbb{E}(f|\mathcal{B}')\|_2^2-\|\mathbb{E}(f|\mathcal{B})\|_2^2=\|\mathbb{E}(f|\mathcal{B}')-\mathbb{E}(f|\mathcal{B})\|_2^2=\|f_\mathrm{sml}\|_2^2>\epsilon^2
  \end{equation*}
  and therefore each iteration increases the value of $\|\mathbb{E}(f|\mathcal{B})\|_2^2$ by at least $\epsilon^2$, so we cannot have more than $\frac{1}{\epsilon^2}$ iterations in total. It means that in the end the result is true with $M_0=\underbrace{\mathcal{F'}(\ldots(\mathcal{F'}(1))\ldots)}_{\lfloor\frac{1}{\epsilon^2}\rfloor\text{ iterations}}.$
 \end{proof}

 The above version of the arithmetic regularity lemma is not quite satisfactory, as we expect to have a slightly different kind of structure for $f_\mathrm{str}$. Before we explain, how to fix that, let us expliot some properties of Fejer kernel.
 
 \begin{lm}
  Let $d$ and $K$ be positive integers, let $\chi_j:\mathbb{T}^d\to\mathbb{C}$ ($j=1,\ldots,d$) be the basic characters defined as $\chi_j(t)=e^{2\pi it_j}$ and let $\Phi_K:\mathbb{T}\to[0,+\infty)$ be the Fejer kernel of order $K$:
   \begin{equation*}
   \Phi_K(t)=\frac{1}{K^d}\prod_{j=1}^d\left|\sum_{k=0}^{K-1}\chi_j^k(t)\right|^2.
  \end{equation*}
  Then $\int_{\mathbb{T}^d}\Phi_K(t)dt=1$ and moreover
  \begin{equation*}
   \int_{[-\lambda,\lambda]^d}\Phi_K(t)dt\ge1-\frac{d}{4K\lambda^2}.
  \end{equation*}

 \end{lm}

 \begin{proof}
  The first assertion is standard; to prove it one only needs to expand $\Phi_K$ as the linear combination of characters and observe that the trivial character comes with coefficient $1$. To prove the inequality, let us first note that
  \begin{equation*}
   \left|\sum_{k=0}^{K-1}\chi_j^k(t)\right|=\left|\frac{1-\chi_j^K(t)}{1-\chi_j(t)}\right|\le\frac{2}{|1-e^{2\pi it_j}|}\le\frac{1}{2\|t_j\|}.
  \end{equation*}
  Therefore if $d=1$, we have 
  \begin{equation*}
   \int_{[-\lambda,\lambda]}\Phi_K(t)dt=1-\int_{\|t\|\ge\lambda}\Phi_K(t)dt\ge 1-\sup_{\|t\|\ge\lambda}\Phi_K(t)=1-\frac{1}{4K\lambda^2}
  \end{equation*}. 
  Now for $d>1$, the $d$-dimensional Fejer kernel is just a product of $d$ copies of a $1$-dimensional one, which gives the bound
  \begin{equation*}
   \int_{[-\lambda,\lambda]^d}\Phi_K(t)dt\ge\left(1-\frac{1}{4K\lambda^2}\right)^d\ge 1-\frac{d}{4K\lambda^2}.
  \end{equation*}
 \end{proof}

 Before stating the next result let us set a default norm on $\mathbb{T}^d$ to be the maximum norm, i.e. $\|t\|=\max_{1\le j\le d}\|t_j\|$. Consequently, we call a function $F:\mathbb{T}^d\to\mathbb{C}$ $M$-Lipschitz, if $|F(t_1)-F(t_2)|\le M\|t_1-t_2\|$ holds for all $t_1,t_2\in\mathbb{T}$. 
 
 \begin{prp}[arithmetic regularity lemma: intermediate version]
  Let $\epsilon>0$ and let $\mathcal{F}$ be a growth function. Then there exixts a real number $M_0=M_0(\epsilon,\mathcal{F})>0$ for which the following statement is true. Let $p>p_0(\epsilon,\mathcal{F})$ be a prime and $f:\mathbb{Z}/p\mathbb{Z}\to [0,1]$ be a function. Then there exists a number $0<M\le M_0$ and a decomposition
  \begin{equation*}
   f=f_\mathrm{str}+f_\mathrm{sml}+f_\mathrm{unf}
  \end{equation*}
  satisfying the following properties:
  \begin{itemize}
   \item $f_\mathrm{str}=F\circ\phi$, where $\phi:\mathbb{Z}/p\mathbb{Z}\to\mathbb{T}^d$ is a homomorphism with $d\le M$ and $F:\mathbb{T}^d\to [0,1]$ is an $M$-Lipschitz function ($f_\mathrm{str}$ is structured),
   \item $\|f_\mathrm{sml}\|_2\le\epsilon$ ($f_\mathrm{sml}$ is small),
   \item $|\langle f_\mathrm{unf},\chi\rangle|\le\frac{1}{\mathcal{F}(M)}$ for every character $\chi$ ($f_\mathrm{unf}$ is $U^2$-uniform).
   \item $f_\mathrm{str}$ and $f_\mathrm{str}+f_\mathrm{sml}$ both take values in $[0,1]$.
   \end{itemize}
 \end{prp}

 \begin{proof}
  First we apply to $f$ the baby version with some different parameters $\epsilon'$ and $\mathcal{F}'$ to be specified later. We get a decomposition $f=f_\mathrm{str}'+f_\mathrm{sml}'+f_\mathrm{unf}'$; now we set $f_\mathrm{unf}=f_\mathrm{unf}'$ and try to find $f_\mathrm{str}$ of the new type so that $\|f_\mathrm{str}-f_\mathrm{str}'\|_2$ is small and in the end set $f_\mathrm{sml}=f_\mathrm{sml}'+f_\mathrm{str}'-f_\mathrm{str}$. We know the structure of $f_\mathrm{str}'$; it can be alternatively said that $f_\mathrm{str}'=F'\circ\phi$, where $\phi=(\phi_1,\ldots,\phi_{d}):\mathbb{Z}/p\mathbb{Z}\to\mathbb{T}^{d}$ is just a product of all pre-characters forming $\Gamma$ and $F':\mathbb{T}^{d}\to [0,1]$ is a function that is constant on the cubes of the form $[\frac{k_1}{n},\frac{k_1+1}{n}]\times\ldots\times[\frac{k_d}{n},\frac{k_d+1}{n}]$ for $k_1,\ldots,k_d\in\mathbb{Z}$. The function $F'$ does not need to be unique, we can pick any that fits into the formula. Also, it does not need to be Lipschitz and we 
have to fix that. To do this, put $F=F'*\Phi_K$ for some $K$ to be specified. To see that $F$ is Lipschitz, let us calculate
  \begin{multline*}
   |F(t_1)-F(t_2)|=\left|\int_{\mathbb{T}^d}F'(s)(\Phi_K(t_1-s)-\Phi_K(t_2-s))\right|\le\\
   \le\sup_s|\Phi_K(t_1-s)-\Phi_K(t_2-s)|.
  \end{multline*}
  The above calculation shows that the Lipschitz constant of $F$ is bounded by that of $\Phi_K$. To estimate it, let us note that $\Phi_K$ is a linear combination of characters of the form $\prod_{j=1}^d\chi_j^{k_j}$ with $|k_j|<K$, and a character of this particular form comes with coefficient $\prod_{j=1}^d(1-\frac{|k_j|}{K})$ and is itself a $(2\pi\sum_{j=1}^d|k_j|)$-Lipschitz function. Therefore the Lipschitz constant of $\Phi_K$ is at most
  \begin{equation*}
   L_d=2\pi\sum_{k_1,\ldots,k_d}\left(\left(\sum_{j=1}^d|k_j|\right)\cdot\prod_{j=1}^d\left(1-\frac{|k_j|}{K}\right)\right).
  \end{equation*}
  To calculate this, let us set $L_d'=\sum_{k_1,\ldots,k_d}\prod_{j=1}^d\left(1-\frac{|k_j|}{K}\right)$. Then one can check that those sequenced satisfy the recurrence $L_{d_1+d_2}=L_{d_1}L_{d_2}'+L_{d_2}L_{d_1}'$ and $L'_{d_1+d_2}=L'_{d_1}+L'_{d_2}$, which together with the boundary conditions $L_1=\frac{2\pi}{3}(K^2-1)$, $L_1'=K$ gives $L_d=\frac{2\pi}{3}dK^{d-1}(K^2-1)\le 4d K^{d+1}$.
  
  Set $f_\mathrm{str}=F\circ\phi$. We would like to bound the expression
  \begin{equation*}
   \|f_\mathrm{str}'-f_\mathrm{str}\|_2^2=\mathbb{E}_x|F'(\phi(x))-F'*\Phi_K(\phi(x))|^2.
  \end{equation*}
  Inside the expectation, some of the elements $s$ will lie near the edges of the cubes and for them it would be hard to estimate the value $|F'(\phi(x))-F'*\Phi_K(\phi(x))|$ other than trivially by $1$. Let us estimate the number of such ``bad'' elements: the set of all $t_j\in\mathbb{T}$ with $\|t_j-\frac{k_j}{n}\|\le\lambda$ has measure $2\lambda$; since $p$ is sufficiently large we can assume that the set of all $x$ with $\|\phi_j(x)-\frac{k_j}{n}\|\le\lambda$ has size at most $4\lambda p$. Taking into account all possible values of $j$ and $k_j$ we see that all but at most $4\lambda dnp$ elements are separated from the boundary of their cubes by at least $\lambda$. For those elements $x$ let us estimate
  \begin{equation*}
   |F'(\phi(x))-F'*\Phi_K(\phi(x))|\le\int_\mathbb{T}\Phi_K(t)|F'(\phi(x))-F'(\phi(x)-t)|dt.
  \end{equation*}
  By the description of $x$ the latter factor is zero on the cube $[-\lambda,\lambda]^d$; on the remaining set it is bounded by $1$, so by the previous lemma the value of the integral is bounded by $\frac{d}{4K\lambda^2}$. In the end, taking into account all values of $x$, we have an estimate
  \begin{equation*}
   \|f_\mathrm{str}'-f_\mathrm{str}\|_2^2=\mathbb{E}_x|F'(\phi(x))-F'*\Phi_K(\phi(x))|^2\le 4\lambda dn+\left(\frac{d}{4K\lambda^2}\right)^2.
  \end{equation*}
  We wish the last quantity to be at most $\frac{\epsilon}{2}$; to achieve this set $\lambda=\frac{\epsilon}{16dn}$ and $K=\lceil\frac{d}{2\lambda^2\sqrt{\epsilon}}\rceil$.
  
  Now we return to the beginning, where we had to specify $\epsilon'$ and $\mathcal{F}'$. We can take $\epsilon'=\frac{\epsilon}{2}$, then $\|f_\mathrm{sml}\|_2\le\|f_\mathrm{sml}'\|_2+\|f_\mathrm{str}'-f_\mathrm{str}\|_2\le\frac{\epsilon}{2}+\frac{\epsilon}{2}=\epsilon$. To choose $\mathcal{F}'$, let us first note that \begin{equation*}
    K\le\frac{d}{\lambda^2\sqrt{\epsilon}}=\frac{2^8d^3n^2}{\epsilon^{5/2}}\le\frac{2^8M^5}{\epsilon^{5/2}}.                                                                                                                                                                                                                                                                                                                            
  \end{equation*}
  The Lipschitz constant of $F$ is then bounded by
  \begin{equation*}
   4d K^{d+1}\le 4M\left(\frac{2^8M^5}{\epsilon^{5/2}}\right)^{M+1}=:a(M,\epsilon).
  \end{equation*}
  It is now enough to take $\mathcal{F}'(M)=\mathcal{F}(a(M,\epsilon))$ and $M_0=M_0(\epsilon',\mathcal{F}')$ given by the previous version of the lemma.
 \end{proof}

 Now the structure of $f_\mathrm{str}$ appears to be more natural, although we are still missing some information. We would like to know that the image of the homomorphism $\phi$ is well equidisributed in $\mathbb{T}^d$ so that we could expect that $f_\mathrm{str}$ has roughly the same global structure as $F$. To achieve this, let us set a notion of $K$-independence. The homomorphism $\phi=(\phi_1,\ldots,\phi_d)$ will be called $K$-independent if the only solution to the equation $k_1\phi_1+\ldots+k_d\phi_d=0$ with $|k_j|<K$ is $k_1=\ldots=k_d=0$. We will show how we can require independence, in particular what to do if $\phi$ turns out not to be independent.
 
 \begin{lm}\label{matrix}
  Let $(a_1,\ldots a_d)$ be a vector with integer coordinates. There exists a matrix $A=[a_{ij}]\in \mathcal{M}_d(\mathbb{Z})$ with $a_{1j}=a_j$ (for $j=1,\ldots,d$) and satisfying the following properties:
  \begin{align*}
   \sum_{j=1}^da_{j}a_{ij}&=0\qquad\text{for }i=2,\dots,d,\\
   \det A&=\frac{\sum_{j=1}^da_j^2}{\gcd(a_1,\ldots,a_d)}.
  \end{align*}
 Moreover, if $\displaystyle{\max_{1\le j\le d}}|a_j|\le K$, then we can choose the entries of the matrix $A$ to be bounded by $K$. 
 \end{lm}

 \begin{proof}
  Let us prove the claim by induction on $d$. For $d=1$ the statement is trivial. Suppose $d>1$ and we have already proved it for $d-1$. We would like to extend the matrix found for the vector $(a_1,\ldots,a_{d-1})$ to make it work for $(a_1,\ldots,a_d)$. Setting $a_{id}=0$ for $i=2,\ldots,d-1$ makes the first property satisfied for those values of $i$. Also, it makes the determinant quite easy to calculate by expanding it with respect to the last column. Since
  \begin{equation*}
   \frac{\sum_{j=1}^da_j^2}{\gcd(a_1,\ldots,a_d)}=\frac{\sum_{j=1}^{d-1}a_j^2}{\gcd(a_1,\ldots,a_{d-1})}\cdot\frac{\gcd(a_1,\ldots,a_{d-1})}{\gcd(a_1,\ldots,a_d)}+\frac{a_d^2}{\gcd(a_1,\ldots,a_d)},
  \end{equation*}
  it looks reasonable to set $a_{dd}=\frac{\gcd(a_1,\ldots,a_{d-1})}{\gcd(a_1,\ldots,a_d)}\in\mathbb{Z}$ and try to complete the last row so that the submatrix $B$ obtained by deleting the first row and the last column has determinant $\frac{(-1)^{d-1}a_d}{\gcd(a_1,\ldots,a_d)}\in\mathbb{Z}$. Note that the formula
  \begin{equation*}
   \det B=\frac{(-1)^{d}\sum_{j=1}^{d-1}a_ja_{dj}}{\gcd(a_1,\ldots,a_{d-1})}
  \end{equation*}
  is true if we set $a_{dj}=a_{ij}$ for some $1\le i\le d-1$ and all $j=1,\ldots d-1$. Since the vectors $(a_{ij})_{j=1}^{d-1}$ span all of $\mathbb{R}^{d-1}$ (as the determinant of the matrix they form is non-zero by the inductive hypothesis), the formula above is in fact true for any choice of $a_{d,1},\ldots, a_{d,d-1}$. This is good for us --- if we insist that
  \begin{equation*}
   0=\sum_{j=1}^da_ja_{dj}=\sum_{j=1}^{d-1}a_ja_{dj}+\frac{a_d\gcd(a_1,\ldots,a_{d-1})}{\gcd(a_1,\ldots,a_d)},
  \end{equation*}
  then automatically we have
  \begin{multline*}
   \det A=a_{dd}\cdot\frac{\sum_{j=1}^{d-1}a_j^2}{\gcd(a_1,\ldots, a_{d-1})}+(-1)^{d-1}a_d\cdot\det B=\\
   \frac{\sum_{j=1}^{d-1}a_j^2}{\gcd(a_1,\ldots, a_d)}-\frac{a_d\sum_{j=1}^{d-1}a_ja_{dj}}{\gcd(a_1,\ldots,a_{d-1})}=\frac{\sum_{j=1}^{d}a_j^2}{\gcd(a_1,\ldots, a_d)}.
  \end{multline*}
  So the only condition remaining is $\sum_{j=1}^{d-1}a_ja_{dj}=\frac{-a_d}{\gcd(a_1,\ldots,a_d)}\cdot\gcd(a_1,\ldots,a_{d-1})$. This can be satisifed by the Euclidean algorithm since $\frac{-a_d}{\gcd(a_1,\ldots,a_d)}\in\mathbb{Z}$, and thus we have proved the existence of the matrix $A$. 

  Now let us consider the bounds for the entries. Obviously if $|a_j|\le K$, then $|a_{dd}|=\left |\frac{\gcd(a_1,\ldots,a_{d-1})}{\gcd(a_1,\ldots,a_d)}\right |\le K$. Choosing the vector $(a_{d,1},\ldots,a_{d,d-1})$ carefully might be a little bit more complicated. But if we write the equation in the form
  \begin{equation*}
   \sum_{j=1}^{d-1}a_{dj}\cdot \frac{a_j}{\gcd(a_1,\ldots,a_{d-1})}=-\frac{a_d}{\gcd(a_1,\ldots,a_d)},
  \end{equation*}
  we can see that the claim boils down to the lemma below.
 \end{proof}

 \begin{lm}
  Let $m,K>0$ be integers and let $b_1,\ldots,b_m$ be coprime integers not exceeding $K$ in absolute value. Let $b$ be an integer with $|b|\le K$. Then there exist integers $c_1,\ldots, c_m$ not exceeding $K$ in absolutee value and satisfying
  \begin{equation*}
   b_1c_1+\ldots+b_mc_m=b.
  \end{equation*}
 \end{lm}

 \begin{proof}
  If $m=1$, then $b_1=\pm 1$ and thee statement is trivial. For $m=2$, if either of $b_1,b_2$ is equal to $\pm 1$, the statement is trivial as well. If it is not the case, then without loss of generality we can assume $b_1>b_2>0$. But then the numbers $b-kb_2$ for $|k|\le K$ are at most $Kb_1$ in absolute value and at least one of them is a multiple of $b_1$, so the claim follows. Suppose now that $m>2$ and we have already proven the claim for all smaller values of $m$. Let $g=\gcd(b_1,\ldots,b_{m-1})$. Then $b_m$ and $g$ are coprime and have absolute value at most $K$, so we can use the claim for $m=2$ to get $gb'+b_mc_m=b$ with $|b'|,|c_m|\le K$. Also the numbers $\frac{b_1}{g},\ldots,\frac{b_{m-1}}{g}$ are coprime integers bounded by $K$ in absolute value, which gives us $\frac{b_1}{g}c_1+\ldots+\frac{b_{m-1}}{g}c_{m-1}=b'$ with $|c_1|,\ldots,|c_{m-1}|\le K$. Combining those two identities we get the claim.
 \end{proof}
 
 Lemma \ref{matrix} gives us the following corollary.
 
 \begin{cor}\label{reduction}
  Let $d>1$ be an integer, let $\phi:\mathbb{Z}/p\mathbb{Z}\to\mathbb{T}^d$ be a homomorphism and let $F:\mathbb{T}^d\to\mathbb{C}$ be an $M$-Lipschitz function. Then at least one of the following holds:
  \begin{itemize}
   \item $\phi$ is $K$-independent,
   \item there exits a homomorphism $\phi':\mathbb{Z}/p\mathbb{Z}\to\mathbb{T}^{d-1}$ and a $dKM$-Lipschitz function $F':\mathbb{T}^{d-1}\to\mathbb{C}$ with $F'\circ\phi'=F\circ\phi$.
  \end{itemize}
 \end{cor}

 \begin{proof}
  Suppose $\phi$ is not $K$-independent, i.e. there exist intgers $a_1,\ldots,a_d$ with $\sum_{j=1}^da_j\phi_j=0$ and $|a_j|<K$. It is not hard to see that the second part is true for $p\le K$ as long as $d-1\ge 1$; suppose then $p>K$. In that case we are allowed to divide all of $a_j$ by their greatest common divisor and without loss of generality assume $\gcd(a_1,\ldots,a_d)=1$. By Lemma \ref{matrix} we can find $d-1$ integer vectors orthogonal to $a=(a_1,\ldots,a_d)$ suth that the matrix $A$ consisting of all of them has determinant $\sum_{j=1}^d a_j^2$. We claim that the $\mathbb{Z}$-span of these $d-1$ vectors coincides with the intersection of their $\mathbb{R}$-span and $\mathbb{Z}^d$. Indeed, we know that the $\mathbb{Z}$-span of all $d$ vectors is a subgroup of $\mathbb{Z}^d$ of index $\det A=\sum_{j=1}^da_j^2$. On the other hand, the map $x\mapsto\langle a,x\rangle\pmod{\det A}$ gives rise to a surjective homomorphism from the quotient group to a group of size $\det A$. Therefore this homomorphism 
is in fact an isomorphism and its kernel is precisely the $\mathbb{Z}$-span of all $d$ vectors. Intersecting it with the $\mathbb{R}$-span of the last $d-1$ vecrors obviously gives us their $\mathbb{Z}$-span. But this intersection can be easily seen as $\{x\in\mathbb{Z}^d:\langle a,x\rangle=0\}$ or in other words the intersection of the $\mathbb{R}$-span and all of $\mathbb{Z}^d$.
  
  Since $\gcd(a_1,\ldots,a_d)=1$, it follows that the set $\{t\in\mathbb{T}^d:\langle a,t\rangle=0\in\mathbb{T}\}$ is in fact the image of the subspace $\{x\in\mathbb{R}^d:\langle a,x\rangle=0\in\mathbb{R}\}$ under the projection $\pmod{\mathbb{Z}^d}$. Therefore it can be parametrized as $A'(\mathbb{T}^{d-1})$, where $A'$ is $(d-1)\times d$ matrix obtained from $A$ by deleting its first row $(a_1,\ldots, a_n)$. Let $F':\mathbb{T}^{d-1}\to\mathbb{C}$ and $\phi':\mathbb{Z}/p\mathbb{Z}\to\mathbb{T}^{d-1}$ be functions satisfying $F'=F\circ A'$ and $\phi=A'\circ\phi'$. Note that $\phi'$ is well defined and is a homomorphism. Then
  \begin{equation*}
   F'\circ\phi'=F\circ A'\circ\phi'=F\circ\phi.
  \end{equation*}
  The only thing left is to estimate the Lipschitz constant of $F'$. Since $A'$ has entries bounded by $K$, it can be viewed as a $dK$-Lipshitz function. Composing it with an $M$-Lipshitz function $F$ gives us a function of Lipschitz constant at most $dKM$.
 \end{proof}

 Now we are ready to give a proof of the full version of the regularity lemma (in the $U^2$ case).
 
 \begin{thm}[arithmetic regularity lemma: final version]\label{arl}
  Let $\epsilon>0$ and let $\mathcal{F}$ be a growth function. Then there exixts a real number $M_0=M_0(\epsilon,\mathcal{F})>0$ for which the following statement is true. Let $p>p_0(\epsilon,\mathcal{F})$ be a prime and $f:\mathbb{Z}/p\mathbb{Z}\to [0,1]$ be a function. Then there exists a number $0<M\le M_0$ and a decomposition
  \begin{equation*}
   f=f_\mathrm{str}+f_\mathrm{sml}+f_\mathrm{unf}
  \end{equation*}
  satisfying the following properties:
  \begin{itemize}
   \item $f_\mathrm{str}=F\circ\phi$, where $\phi:\mathbb{Z}/p\mathbb{Z}\to\mathbb{T}^d$ is a $\mathcal{F}(M)$-independent homomorphism with $d\le M$ and $F:\mathbb{T}^d\to [0,1]$ is an $M$-Lipschitz function ($f_\mathrm{str}$ is structured),
   \item $\|f_\mathrm{sml}\|_2\le\epsilon$ ($f_\mathrm{sml}$ is small),
   \item $|\langle f_\mathrm{unf},\chi\rangle|\le\frac{1}{\mathcal{F}(M)}$ for every character $\chi$ ($f_\mathrm{unf}$ is $U^2$-uniform).
   \item $f_\mathrm{str}$ and $f_\mathrm{str}+f_\mathrm{sml}$ both take values in $[0,1]$.
   \end{itemize}
 \end{thm}

 \begin{proof}
  Let us start with applying the previous version of the regularity lemma with the same parameter $\epsilon$ and a different growth function $\mathcal{F}'$ to be specified. We stick to the obtained decomposition $f=f_\mathrm{str}+f_\mathrm{sml}+f_\mathrm{unf}$ but we would like to exploit more properties of $F$ and $\phi$. If $\phi$ is $\mathcal{F}$-independent, we are done. Otherwise we use Lemma \ref{reduction} to decrease the dimension $d$ by $1$ at the cost of potentially increasing the Lipschitz constant up to $dM\mathcal{F}(M)\le M^2\mathcal{F}(M)$. Put $\mathcal{F}_1(M)=M^2\mathcal{F}(M)$; since this procedure can be applied at most $d\le M$ times, so the correct choice of $\mathcal{F}'$ is $\mathcal{F}'(M)=\mathcal{F}(\underbrace{\mathcal{F}_1(\ldots(\mathcal{F}_1(M))\ldots)}_{\lfloor M\rfloor\text{ times}})$.
 \end{proof}

 To make a proper use of the above result, we often need to relate the behaviours of $f_\mathrm{str}$ and $F$. Below we prove a statement of this kind.
 
 \begin{lm}
  Let $d$ be a positive integer, $p$ be a prime, let $\phi_1,\ldots, \phi_d:\mathbb{Z}/p\mathbb{Z}\to\mathbb{T}$ be pre-characters and let $\phi=(\phi_1,\ldots,\phi_d):\mathbb{Z}/p\mathbb{Z}\to\mathbb{T}^d$ be their product. Suppose that the set $\{\phi_j\}_{j=1}^d$ is $K$-independent. Let $F:\mathbb{T}^d\to\mathbb{C}$ be an $M$-Lipschitz function. Then
  \begin{equation*}
   \left|\mathbb{E}_x F(\phi(x))-\int_{\mathbb{T}^d}F(t)dt\right|\le\frac{M}{\sqrt{K}}.
  \end{equation*}
 \end{lm}

 \begin{proof}
  Let $\chi_j:\mathbb{T}^d\to\mathbb{C}$ defined via $\chi(t_1\ldots,t_d)=e^{2\pi i t_j}$ be the basic characters and let $\Phi_K:\mathbb{T}\to\mathbb{C}$ be the Fejer kernel defined via the formula
  \begin{equation*}
   \Phi_K(t)=\frac{1}{K^d}\prod_{j=1}^d\left|\sum_{k=0}^{K-1}\chi_j^k(t)\right|^2.
  \end{equation*}
  Recall that $\int_{\mathbb{T}^d}\Phi_K(t)dt=1$. We also have the bound
  \begin{equation*}
   \left|\sum_{k=0}^{K-1}\chi_j^k(t)\right|=\left|\frac{1-\chi_j^K(t)}{1-\chi_j(t)}\right|\le\frac{2}{|1-e^{2\pi it_j}|}\le\frac{1}{2\|t_j\|}.
  \end{equation*}
  Therefore $\Phi_K(t)\le\prod_{j=1}^d\frac{1}{4K\|t_j\|^2}$. Having this inequality we would like to show that $F(t)$ and $F*\Phi_K(t)$ are close together for any $t\in\mathbb{T}^d$. Let us estimate their difference:
  \begin{equation*}
   |F(t)-F*\Phi_K(t)|=\left|\int_{\mathbb{T}^d}(F(t)-F(t-s))\Phi_K(s)ds\right|\le M\int_{\mathbb{T}^d}\|s\|\Phi_K(s)ds.
  \end{equation*}
  This is already independent of $t$; now we make use of the fact that $\Phi_K(s)$ is large precisely when $\|s\|$ is small; more accurately $\|s\|\ge u$ implies $\Phi_K(s)\le(\frac{1}{4Ku^2})^d$ for any $u\ge 0$. Combining this with $\int_{\mathbb{T}^d}\Phi_K(t)dt=1$ leads to the inequality
  \begin{multline*}
   \int_{\mathbb{T}^d}\|s\|\Phi_K(s)ds=\int_{\mathbb{T}^d}\int_0^{1/2} 1_{\{u\le\|s\|\}}\Phi_K(s)duds=\\
   =\int_0^{1/2}\int_{\mathbb{T}^d} 1_{\{u\le\|s\|\}}\Phi_K(s)dsdu\le\int_0^{1/2}\min\left(1,\frac{1}{4Ku^2}\right)^ddu\le\frac{1}{\sqrt{K}},
  \end{multline*}
  and consequently $|F(t)-F*\Phi_K(t)|\le\frac{M}{\sqrt{K}}$ for any $t\in\mathbb{T}^d$. In particular, the triangle inequality yields $|\mathbb{E}_xF(\phi(x))-\mathbb{E}_xF*\Phi_K(\phi(x))|\le\frac{M}{\sqrt{K}}$. Let us now expand the last expression:
  \begin{multline*}
   \mathbb{E}_xF*\Phi_K(\phi(x))=\mathbb{E}_x\left(\int_{\mathbb{T}^d}F(t)\Phi_K(\phi(x)-t)dt\right)=\\
   =\int_{\mathbb{T}^d}F(t)\left(\mathbb{E}_x\Phi_K(\phi(x)-t)\right).dt
  \end{multline*}
  Expanding out the formula for $\Phi_K$, we can see that the average $\mathbb{E}_x\Phi_K(\phi(x)-t)$ is a sum of the averages of the form $c\cdot\mathbb{E}_x\chi(\phi(x)-t)$, where $\chi(t)=\prod_{j=1}^d\chi_j^{\alpha_j}$ for some $\alpha_1,\ldots\alpha_j\in\{1-K,\ldots, -1,0,1,\ldots,K-1\}$. But from the $K$-independence of $\phi$ we can see that $\chi(\phi(x)-t)$ is never a constant function in $x$ and therefore has average $0$, unless $\chi$ is a trivial chatracter with $\alpha_1=\ldots=\alpha_j=0$, in which case $c=1$. In the end, $\mathbb{E}_x\Phi_K(\phi(x)-t)=1$ for all $t\in\mathbb{T}^d$, which leads to $\mathbb{E}_xF*\Phi_K(\phi(x))=\int_{\mathbb{T}^d}F(t)dt$. Plugging this formula into the previous inequality, we get the desired result.
 \end{proof}

 One may wonder how we managed to prove the above lemma without assuming that $p$ is large. In fact, the $K$-independence of a set of $d$ homomorphisms carries a hidden assumption $p\ge K^d$.
 
 In this paper we did not need the above result, or even the full version of the regularity lemma (the intermediate version would be enough), but in general it might be useful to have them around.

\end{document}